\documentclass[11pt]{amsart}

\usepackage[english]{babel}

\usepackage[letterpaper,top=3.2cm,bottom=3.2cm,left=3.3cm,right=3.3cm,marginparwidth=1.8cm]{geometry}

\usepackage{amsmath,amssymb,graphicx}
\usepackage{setspace}
\usepackage{epigraph}
\usepackage{mathtools}
\numberwithin{equation}{section}

\newtheorem{theorem}{Theorem}[section]
\newtheorem{lemma}{Lemma}[section]

\newtheorem{remark}{Remark}[section]

\title{On Abel's Identity}
\author{Mehrzad Ajoodanian}

\begin{document}

\maketitle

\epigraph{\emph{\small "Study the masters, not the pupils."}}{\small --- Niels Henrik Abel (1802--1829)}

\begin{abstract}
We provide a natural duality that matches, in reverse order, the coefficients of the characteristic polynomial of the Maurer-Cartan of the Wronskian matrix with the coefficients of the original differential equation. Abel's identity is recovered as a corollary.  
\end{abstract}
\let\thefootnote\relax\footnote{I would like to thank Amir Jafari.}
\section{Introduction}
Abel's identity remains one of the most elegant and inspiring results in the theory of ordinary differential equations. Although Abel passed away at the young age of twenty-seven, his ideas have stood the test of time. This paper offers a natural duality that generalizes Abel's identity. 

Let $I$ be an interval in $\mathbb{R}$, and let $V$ be an $n$-dimensional real vector space. Suppose $A \colon I \to V$ is a smooth map. The \emph{Wronskian matrix} of $A$ relative to a chosen basis is defined by
\[
W(A) =
\begin{bmatrix}
 a_1 & a_2 & \cdots & a_n \\
 a_1' & a_2' & \cdots & a_n' \\
 \vdots & \vdots & \ddots & \vdots \\
 a_1^{(n-1)} & a_2^{(n-1)} & \cdots & a_n^{(n-1)}
\end{bmatrix},
\]
where $A = (a_1, \dots, a_n)$ in the given basis, and $a_i^{(j)}$ denotes the $j$-th derivative of $a_i$. When no confusion arises, we simply write $W = W(A)$.

\begin{lemma}
Let $T$ denote a change of basis in $V$. Then the Wronskian matrix transforms as
\[
W(A) \mapsto W(A)T.
\]
\end{lemma}

The \emph{Wronskian determinant} of $A$ is then defined by
\[
 w(A) = \det W(A).
\]
We assume throughout that $w(A) \neq 0$ on $I$. Since $V$ is $n$-dimensional, $A$ satisfies a linear homogeneous differential equation of order $n$:
\[
 A^{(n)} = p_1 A^{(n-1)} + p_2 A^{(n-2)} + \cdots + p_n A,
\]
where the coefficients $p_i$ are smooth functions on $I$. These coefficients are independent of the choice of basis in $V$. By Cramer's rule, we may express them as
\[
 p_i = \frac{w_i(A)}{w(A)},
\]
where $w_i(A)$ denotes the determinant obtained by replacing the $(n+1-i)$-th row of $W(A)$ with $(a_1^{(n)}, \dots, a_n^{(n)})$.

Abel’s identity asserts that the Wronskian determinant $w$ satisfies the first-order differential equation
\[
 p_1 = \frac{w'}{w}.
\]
Abel's identity follows from Cramer's rule together with the standard formula for the derivative of a determinant, expressed as the sum of determinants obtained by differentiating each row in turn.

A natural question then arises: what role do the remaining coefficients $p_i$ play? This paper provides an answer via a natural duality.

\section{The Maurer--Cartan Form}
In the context of Lie groups and gauge theory, the \emph{Maurer--Cartan forms} for a Lie group $G$ are defined by
\[
 R(g) = dg\,g^{-1}, \qquad L(g) = g^{-1}dg,
\]
are $1$-forms with values in the Lie algebra of $G$. Here we adopt the right Maurer--Cartan form, as it proves more convenient for our purposes. On an interval $U \subset \mathbb{R}$, we have the advantage of global coordinates, so we may define the \emph{Maurer--Cartan matrix} associated to $A$ by
\[
 R = W'W^{-1}.
\]

\begin{lemma}
The matrix $R$ is independent of the choice of basis in $V$.
\end{lemma}

\begin{proof}
Suppose $T$ is a change of basis for $V$. Then $W(TA) = W(A)T$, hence
\[
 R(TA) = (W(A)T)'(W(A)T)^{-1} = W'(A)TT^{-1}W(A)^{-1} = R(A).
\]
\end{proof}

Note that $R$ is invariant under a change of basis, whereas the left Maurer--Cartan form $L = W^{-1}W'$ transforms by conjugation under such changes.

Let $q_j$ denote the coefficients of the characteristic polynomial of $R$. By the Cayley--Hamilton theorem, $R$ satisfies its characteristic polynomial:
\[
 R^n = q_1 R^{n-1} + q_2 R^{n-2} + \cdots + q_n I,
\]
where each $q_j$ is a smooth function on $I$ and independent of the choice of basis.

\begin{remark}
We prefer $R$ to $L$ because $R$ is basis-invariant as a matrix. However, since $R$ and $L$ are conjugate, their characteristic polynomials coincide. Hence, the coefficients $q_i$ may equivalently be read from the characteristic polynomial of $L$.
\end{remark}

\section{Main Theorem}
We now present the natural duality behind Abel’s identity.

\begin{theorem}
The coefficients of the characteristic polynomial of Maurer-Cartan of the Wronskian matrix are, in reverse order, the same as the coefficients of the original differential equation. More precisely,
for all $0 < i < n$, the following equality holds:
\[
 q_i = p_j, \qquad \text{whenever } i + j = n + 1.
\]
\end{theorem}

As an immediate corollary, we recover Abel’s identity.
\begin{theorem}[Abel]
\[
 p_1 = \frac{w'}{w}.
\]
\end{theorem}

\begin{proof}\[
 p_1 = q_n = \det(R) = \det(W'W^{-1}) = \frac{\det(W')}{\det(W)} = \frac{(\det W)'}{\det(W)} = \frac{w'}{w}.
\]
\end{proof}

\section{Proof of the Main Theorem}
We prefer to work with $R$ since it is independent of the choice of basis. We claim that $R$ can be decomposed uniquely as a sum $R = a+ b$, where:
\begin{itemize}
\item $a$ is a constant matrix whose only nonzero entries occur for $j - i = 1$ and are equal to one.
  \item $b$ is a rank-one matrix whose entries vanish except in the last row ($i = n$), where $E_{nj} = p_i$ whenever $i + j = n + 1$;  
\end{itemize}
Instead, we show that
\[
 W' = (a + b)W.
\]
The product $bW$ reproduces the last row of $W'$, while $aW$ shifts each row of $W$ upward by one, precisely matching the structure of $W'$. This establishes the claim about $R$. The coefficients of the characteristic polynomial of $R$ then yield $q_i = p_j$ whenever $i + j = n + 1$.
\qed

%\section*{References}

\bigskip
Mehrzad Ajoodanian\\
\noindent\textit{Email:} {mehrzad77@gmail.com}

\end{document}